\documentclass{article}[12pt]

\usepackage{amsthm}
\usepackage{mathrsfs}
\usepackage{amsmath}
\usepackage{amsfonts}
\usepackage{enumerate}
\usepackage{amssymb}
\usepackage{dsfont}
\usepackage{color}
\usepackage{mathrsfs}
\usepackage{systeme}

\usepackage{float}
\usepackage{verbatim}

\usepackage{emptypage}

\usepackage{hyperref}                                                                                                                 
 \hypersetup{colorlinks,citecolor=blue,filecolor=blue,linkcolor=blue,urlcolor=blue}
\usepackage[pdftex]{graphicx}
\usepackage{float}
\usepackage{enumitem}
\usepackage[natbibapa]{apacite}

\def\P{\mathbb{P}}
\def\R{\mathbb{R}}
\def \F {\mathcal{F}}
\def\E{\mathbb{E}}
\def\N{\mathbb{N}}

\def \1nd{\mathds{1}}

\def \I{\mathbb{I}}
\DeclareMathOperator*{\essinf}{ess\,inf}

\newtheorem{deff}{Definition}[section]
\newtheorem{thm}[deff]{Theorem}

\newtheorem{lemma}[deff]{Lemma}

\newtheorem{rem}[deff]{Remark}

\title{Predicting the Last Zero of a Spectrally Negative L\'evy Process}

\begin{document}

\author{Erik J. Baurdoux\footnote{Department of Statistics, London School of Economics and Political Science. Houghton street, {\sc London, WC2A 2AE, United Kingdom.} E-mail: e.j.baurdoux@lse.ac.uk} \quad \& \quad J. M. Pedraza  \footnote{Department of Statistics, London School of Economics and Political Science. Houghton street, {\sc London, WC2A 2AE, United Kingdom.} E-mail: j.m.pedraza-ramirez@lse.ac.uk} }
\date{\today}
\maketitle

\begin{abstract}
\noindent
Last passage times arise in a number of areas of applied probability, including risk theory and degradation models. Such times are obviously not stopping times since they depend on the whole path of the underlying process. 
We consider the problem of finding a stopping time that minimises the $L^1$-distance to the last time a spectrally negative  L\'evy process $X$ is below zero. Examples of related problems in a finite horizon setting for processes with continuous paths are \cite{du2008predicting} and \cite{glover2014optimal}, where the last zero is predicted for a Brownian motion with drift, and for a transient diffusion, respectively. 

As we consider the infinite horizon setting, the problem is interesting only when the L\'evy process drifts to $\infty$ which we will assume throughout. Existing results allow us to rewrite the problem as a classic optimal stopping problem, i.e. with an adapted payoff process. We use a direct method to show that an optimal stopping time is given by the first passage time above a level defined in terms of the median of the convolution with itself of the distribution function of $-\inf_{t\geq 0}X_t$. We also characterise when continuous and/or smooth fit holds.\end{abstract}

\noindent
{\footnotesize Keywords: L{\'{e}}vy processes, optimal prediction, optimal stopping.}

\noindent
{\footnotesize Mathematics Subject Classification (2000): 60G40, 62M20}

\section{Introduction}
In recent years last exit times have been studied in several areas of applied probability, e.g. in risk theory (see \cite{chiu2005passage}). Consider the Cram\'er--Lundberg process, which is a process consisting of a deterministic drift plus a compound Poisson process which has only negative jumps (see Figure \ref{Cramerlundberg}) which typically models the capital of an insurance company. A key quantity of interest is the time of ruin $\tau_0$, i.e. the first time the process becomes negative. Suppose the insurance company has funds to endure negative capital for some time. Then another quantity of interest is the last time $g$ that the process is below zero. In a more general setting we may consider a spectrally negative L\'evy process instead of the classical risk process. We refer to \cite{chiu2005passage} and \cite{baurdoux2009last} for the Laplace transform of the last time before an exponential time a spectrally negative L\'evy process is below some level.

\begin{figure}[H]
\begin{center} 
\setlength{\unitlength}{.25cm} 
\centering 
\begin{picture}(30,20) 
 
\put(2,0){\vector(0,1){19}} 
\put(2,6){\vector(1,0){24}}
\put(0,18){\small {$X_t$}}
\put(25,4.5){\small{$t$}}

\put(1,9){\tiny{x}}
\put(1.75,9.25){\line(1,0){.5}}

\put(2,9.25){\line(3,4){3}}
\put(5.1,13.3){\circle{.3}}	
\multiput(5.1,13.3)(0,-.5){4}{\line(0,-1){.25}}

\put(5.1,11.5){\line(3,4){2}}
\put(7.2,14.2){\circle{.3}}
\multiput(7.2,14.2)(0,-.5){13}{\line(0,-1){.25}}


\put(7.2,7.9){\line(3,4){2.5}}
\put(9.8,11.3){\circle{.3}}
\multiput(9.8,11.3)(0,-.5){17}{\line(0,-1){.25}}

\put(9.8,5){\small{$\tau_0$}}

\put(15,5){\small{$g$}}

\put(9.8,3){\line(3,4){4.5}}
\put(14.4,9.1){\circle{.3}}
\multiput(14.4,9.1)(0,-.5){8}{\line(0,-1){.25}}

\put(14.4,5.3){\line(3,4){5}}
\put(19.5,12){\circle{.3}}
\multiput(19.5,12)(0,-.5){6}{\line(0,-1){.25}}

\put(19.5,9.2){\line(3,4){3}}

\end{picture}
\end{center} 
\caption{Cram\'er--Lundberg process with $\tau_0$ the moment of ruin and $g$ the last zero.}
\label{Cramerlundberg}
\end{figure}
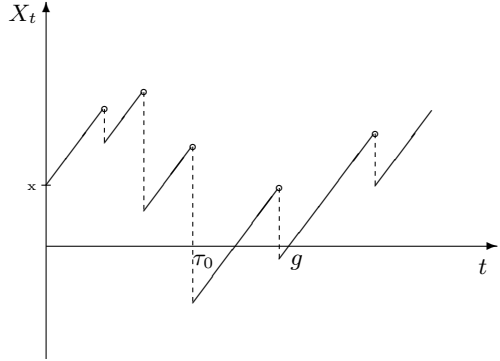

Last passage times also appear in financial modeling. In particular, \cite{madan2008black,madan2008option} showed that the price of a European put and call option for certain non-negative, continuous martingales can be expressed in terms of the probability distributions of last passage times. 

Another application is in degradation models. \cite{paroissin2013first} proposed a spectrally positive L\'evy process as a degradation model. They consider a subordinator perturbed by an independent Brownian motion. The presence of a Brownian motion can model small repairs of the component or system and the jumps represents major deterioration.  Classically, the failure time of a component or system is defined as the first hitting time of a critical level $b$ which represents a failure or a bad performance of the component or system. Another approach is to consider instead the last time that the process is under $b$. Indeed, for this process the paths are not necessarily monotone and hence when the process is above the level $b$ it can return back below it later.\\

The main aim of this paper is to predict the last time a spectrally negative L\'evy process is below zero. More specifically, we aim to find a stopping time that is closest (in $L^1$ sense) to the above random time. This is an example of an optimal prediction problem. Recently, these problems have received considerable attention, for example, \cite{bernyk2011predicting} predicted the time at which a stable spectrally negative L\'evy process attains its supremum in a finite time horizon. A few years later, the infinite time horizon version was solved in \cite{baurdoux2014predicting} for a general L\'evy process with an infinite time horizon. \cite{glover2013three} predicted the time of its ultimate minimum for a transient diffusion processes. \cite{du2008predicting} predicted the last zero of a Brownian motion with drift and \cite{glover2014optimal} predicted the last zero of a transient difussion. It turns out that the problems just mentioned are equivalent to an optimal stopping problem, in other words, optimal prediction problems and optimal stopping problems are intimately related. 

The rest of this paper is organised as follows. In Section \ref{sec_prereqs} we discuss some preliminaries and some technicalities to be used later. Section \ref{sec_main} concerns main result, Theorem \ref{thm:maintheorem}. Section \ref{sec_proof} is then dedicated to the proof of Theorem \ref{thm:maintheorem}. In the final Section we consider specific examples.

\section{Prerequisites and formulation of the problem}\label{sec_prereqs}

Formally, let $X$ be a spectrally negative L\'evy process drifting to infinity, i.e. $\lim_{t\rightarrow \infty} X_t=\infty$, starting from $0$ defined on a filtered probability space $(\Omega,\F, \mathbb{F}, \P)$ where $\mathbb{F}=\{\F_t,t\geq 0 \}$ is the filtration generated by $X$ which is naturally enlarged (see Definition 1.3.38 in \cite{bichteler2002stochastic}). Suppose that $X$ has L\'evy triple $(c,\sigma, \Pi)$ where $c\in \R$, $\sigma\geq 0$ and $\Pi$ is the so-called L\'evy measure concentrated on $(-\infty,0)$ satisfying $\int_{(-\infty,0)} (1\wedge x^2)\Pi(dx)<\infty$. Then the characteristic exponent defined by $\Psi(\theta):=-\log(\E(e^{i \theta X_1}))$ take the form 

\begin{align*}
\Psi(\theta)=ic\theta+\frac{1}{2}\sigma^2 \theta^2 +\int_{(-\infty,0)} (1-e^{i\theta x}+i\theta x \I_{\{ x>-1\}})\Pi(dx).
\end{align*}
Moreover, the L\'evy--It\^o decomposition states that $X$ can be represented as

\begin{align*}
X_t=-ct+\sigma B_t+\int_{[0,t]} \int_{\{x\leq -1 \}} xN(ds,dx)+\int_{[0,t]} \int_{\{x> -1 \}} (xN(ds,dx)-x\Pi(dx)ds),
\end{align*}
where $\{B_t,t\geq 0\}$ is an standard Brownian motion, $N$ is a Poisson random measure with intensity $ds\times \Pi(dx)$ and the process $\{\int_{[0,t]} \int_{\{x> -1 \}} (xN(ds,dx)-x\Pi(dx)ds),t\geq 0\}$ is a square-integrable martingale. Furthermore, it can be shown that all L\'evy processes satisfy the strong Markov property.\\

Let $W^{(q)}$ and $Z^{(q)}$ the scale functions corresponding to the process $X$ (see \cite{kyprianou2013fluctuations} or \cite{bertoin1998levy} for more details). That is, $W^{(q)}$ is such that $W^{(q)}(x)=0$ for $x<0$, and is characterised on $[0,\infty)$ as a strictly increasing and continuous function whose Laplace transform satisfies
\begin{align*}
\int_{0}^{\infty}e^{-\beta x}W^{(q)}(x)dx=\frac{1}{\psi(\beta)-q}\qquad \text{for } \beta>\Phi(q),
\end{align*}
and 
\begin{align*}
Z^{(q)}(x)=1+q\int_0^x W^{(q)}(y)dy
\end{align*}
where $\psi$ and $\Phi$ are, respectively, the Laplace exponent and its right inverse given by

\begin{align*}
\psi(\lambda)&=\log\E(e^{\lambda X_1})\\
\Phi(q)&=\sup\{\lambda \geq 0 : \psi(\lambda)=q \}
\end{align*}
for $\lambda,q\geq 0$.
Note that $\psi$ is zero at zero and tends to infinity at infinity. Moreover, it is infinitely differentiable and strictly convex with $\psi'(0+)=\E(X_1)>0$ (since $X$ drifts to infinity). The latter directly implies that $\Phi(0)=0$.\\

We know that the right and left derivatives of $W$ exist (see \cite{kyprianou2013fluctuations} Lemma 8.2). For ease of notation we shall assume that $\Pi$ has no atoms when $X$ is of finite variation, which guarantees that $W\in C^1(0,\infty)$. Moreover, for every $x\geq 0$ the function $q\mapsto W^{(q)}$ is an analytic function on $\mathbb{C}$.\\

If $X$ is of finite variation we may write 

\begin{align*}
\psi(\lambda)=d \lambda -\int_{(-\infty,0)}(1-e^{\lambda y})\Pi(dy),
\end{align*}
where necessarily 

\begin{align*}
d=-c-\int_{(-1,0)}x\Pi(dx)>0.
\end{align*}
With this notation, from the fact that $0\leq 1-e^{\lambda y}\leq 1$ for $y\leq 0$ and using the dominated convergence theorem we have that 

\begin{align}
\label{eq:expressionphibounded}
\psi'(0+)=d+\int_{(-\infty,0)}x\Pi(dx).
\end{align} 	
For all $q\geq 0$, the function $W^{(q)}$ may have a discontinuity at zero and this depends on the path variation of $X$: in the case that $X$ is of infinite variation we have that $W^{(q)}(0)=0$, otherwise 

\begin{align}
\label{eq:Watzero}
W^{(q)}(0)=\frac{1}{d}.
\end{align}

There are many important fluctuations identities in terms of the scale functions $W^{(q)}$ and $Z^{(q)}$ (see \cite{bertoin1998levy} Chapter VII or \cite{kyprianou2013fluctuations} Chapter 8). We mention some of them that will be useful for us later on. Denote by $\tau_0^-$ the first time the process $X$ is below the zero, i.e.

\begin{align*}
\tau_0^-=\inf \{ t>0: X_t <0\}.
\end{align*}
We then have for $x\in\mathbb{R}$
\begin{align}
\label{eq:tau0finite}
\P_x(\tau_0^-<\infty)=\left\{
\begin{array}{ll}
1-\psi'(0+)W(x) & \text{if } \psi'(0+)\geq 0,\\
1 & \text{if } \psi'(0+)<0,
\end{array}
\right.
\end{align}
where $\P_x$ denotes the law of $X$ started from $x$.

Let us define the $q$-potential measure of $X$ killed on exiting $(-\infty,a)$ for $q\geq 0$ as  follows
\begin{align*}
R^{(q)}(a,x,dy):=\int_0^{\infty} e^{-qt}\P_x(X_t\in dy,\tau_a^+>t).
\end{align*}
The potential measure $R^{(q)}$ has a density $r^{(q)}(a,x,y)$ (see \cite{kyprianou2011theory} Theorem 2.7 for details) which is given by
\begin{align}
\label{eq::qpotentialkillinglessa}
r^{(q)}(a,x,y)=e^{-\Phi(q)(a-x)}W^{(q)}(a-y)-W^{(q)}(x-y).
\end{align}
In particular, $R^{(0)}$ will be useful later. Another pair of processes that will be useful later on are the running supremum and running infimum defined by

\begin{align*}
\overline{X}_t=\sup_{0\leq s \leq t} X_s,\\
\underline{X}_t=\inf_{0\leq s\le t} X_s.
\end{align*}
The well-known duality lemma states that the pairs $(\overline{X}_t, \overline{X}_t-X_t)$ and $(X_t-\underline{X}_t,-\underline{X}_t)$ have the same distribution under the measure $\P$. Moreover,  with $e_q$ an independent exponential distributed random variable with parameter $q>0$, we deduce from the Wiener--Hopf factorisation that the random variables $\overline{X}_{e_q}$ and $\overline{X}_{e_q}-X_{e_q}$ are independent. Furthermore, in the spectrally negative case, $\overline{X}_{e_q}$ is exponentially distributed with parameter $\Phi(q)$. From the theory of scale functions we can also deduce that $-\underline{X}_{e_q}$ is a continuous random variable with
\begin{align}
\label{eq:densityofrunninginfimum}
\P(-\underline{X}_{e_q} \in dz)=\left(\frac{q}{\Phi(q)} W^{(q)}(dz)-qW^{(q)}(z)\right)dz
\end{align}
for $z\geq 0$.

Denote by $g_r$ as the last passage time below $r\geq 0$, i.e.

\begin{align}
\label{eq:lastzero}
g_r=\sup\{t\geq 0:X_t\leq r \}.
\end{align}
When $r=0$ we simply write $g_0=g$.  

\begin{rem}
\label{rem:lastzero}
Note that from the fact that $X$ drifts to infinity we have that $g_r<\infty$ $\P$-a.s. Moreover, as $X$ is a spectrally negative L\'evy process, and hence the case of a compound Poisson process is excluded, the only way of exiting the set $(-\infty,r]$ is by creeping upwards. This tells us that $X_{g_r-}=r$ and that $g_r=\sup\{ t\geq 0: X_t<r\}$ $\P$-a.s.
\end{rem}

Clearly, up to any time $t\geq 0$ the value of $g$ is unknown (unless $X$ is trivial), and it is only with the realisation of the whole process that we know that the last passage time below $0$ has occurred. However, this is often too late: typically one would like to know how close $X$ is to $g$ at any time $t\geq 0$ and then take some action based on this information. We search for a stopping time $\tau_*$ of $X$ that is as ``close'' as possible to $g$. Consider the optimal prediction problem

\begin{align}
\label{eq:optimalprediction}
V_*=\inf_{\tau \in \mathcal{T} } \E|g-\tau|,
\end{align}
where  $\mathcal{T}$ is the set of all stopping times. 

\section{Main result}\label{sec_main}

Before giving an equivalence between the optimal prediction problem \eqref{eq:optimalprediction} and an optimal stopping problem we prove that the random times $g_r$ for $r\geq 0$ have finite mean. For this purpose, let $\tau_x^+$ be the first passage time above $x$, i.e,
\begin{align*}
\tau_x^+=\inf\{t> 0:X_t>x \}.
\end{align*}
%
%
%
%
%
%

\begin{lemma}
\label{cor:moments}
Let $X$ be a spectrally negative L\'evy process drifting to infinity with L\'evy measure $\Pi$ such that 
\begin{align}
\label{eq:assumptionofPi}
\int_{(-\infty,-1)} x^2\Pi(dx)<\infty.
\end{align}
 Then 
$\E_x(g_r)<\infty$ for every $x,r \in \R$. 
\end{lemma}

\begin{proof}
Note that by the spatial homogeneity of L\'evy processes we have to prove that for all $x,r \in \R$.
\begin{align*}
\E_x(g_r)=\E_{x-r}(g)<\infty.
\end{align*} 
Then it suffices to take $r=0$. From \cite{baurdoux2009last} (Theorem 1) or \cite{chiu2005passage} (Theorem 3.1) we know that for a spectrally negative L\'evy process such  that $\psi'(0+)>0$ the Laplace transform of $g$ for $q\geq 0$ and $x\in \R$ is given by

\begin{align*}
\E_x(e^{-q g})=e^{\Phi(q )x}\Phi'(q)\psi'(0+)+\psi'(0+)(W(x)-W^{(q)}(x)).
\end{align*}
Then, from the well-known result which links the moments and derivatives of the Laplace transform (see \cite{feller1971an} (section XIII.2)), the expectation of $g$ is given by 
\begin{align*}
\E_x(g)&=-\frac{\partial}{ \partial q} \E_x(e^{-q g})\bigg|_{q=0+}\\
&=\psi'(0+)\frac{\partial }{ \partial q}W^{(q)}(x)\bigg|_{q=0+}-\psi'(0+)[\Phi''(q)e^{\Phi(q)x}+x\Phi'(q)^2e^{\Phi(q)x}]\bigg|_{q=0+}\\
&=\psi'(0+)\frac{\partial }{ \partial q}W^{(q)}(x)\bigg|_{q=0+}-\psi'(0+)[\Phi''(0)+x\Phi'(0)^2]
\end{align*}

We know that for any $x\in \R$ the function $q\mapsto W^{(q)}$ is analytic, therefore the first term in the last expression is finite. Hence $g$ has finite second moment if $\Phi'(0)$ and $\Phi''(0)$ are finite. Recall that the function $\psi:[0,\infty)\mapsto \R$ is zero at zero and tends to infinity at infinity. Further, it is infinitely differentiable and strictly convex on $(0,\infty)$. Since $X$ drifts to infinity we have that $0<\psi'(0+)<\psi'(\lambda)$ for any $\lambda>0$.
We deduce that $\psi$ is strictly increasing in $[0,\infty)$ and the right inverse $\Phi(q)$ is the usual inverse for $\psi$. From the fact that $\psi$ is strictly convex we have that $\psi''(x)>0$ for all $x>0$.

We then compute
\begin{align*}
\Phi'(0)=\frac{1}{\psi'(\Phi(0)+)}=\frac{1}{\psi'(0+)}<\infty
\end{align*}
and 
\begin{align*}
\Phi''(0)=-\frac{\psi''(\Phi(q)+) \Phi'(q)}{\psi'(\Phi(q)+)^2}\bigg|_{q=0}=-\frac{\psi''(0+)}{\psi'(0+)^3}.
\end{align*}
From the L\'evy--It\^o decomposition of $X$ we know that
\begin{align*}
\psi''(0+)=\sigma^2+\int_{(-\infty,0)}x^2 \Pi(dx)=\sigma^2 +\int_{(-\infty,-1)}x^2 \Pi(dx)+\int_{(-1,0)}x^2 \Pi(dx)<\infty,
\end{align*}
where the last inequality holds by assumption \eqref{eq:assumptionofPi} and from the fact that $\int_{(-1,0)}x^2\Pi(dx)<\infty$ since $\Pi$ is a L\'evy measure. Then we have that $\Phi''(0)>-\infty$ and hence $\E_x(g)<\infty$ for all $x \in \R$.
\end{proof}

Now we are ready to state the equivalence between the optimal prediction problem and an optimal stopping problem mentioned earlier. This equivalence is mainly based on the work of \cite{urusov2005property}. 
\begin{lemma}
Consider the standard optimal stopping problem
\begin{align}
\label{eq:optimalstopping}
V=\inf_{\tau\in \mathcal{T}} \E\left( \int_0^{\tau} G(X_s)ds\right),
\end{align}
where the function $G$ is given by $G(x)=2\psi'(0+)W(x)-1$ for $x\in \R$. Then the stopping time which minimises \eqref{eq:optimalprediction} is the same which minimises \eqref{eq:optimalstopping}. In particular,

\begin{align}
V_*=V+\E(g).
\end{align}
\end{lemma}

\begin{proof}

Fix any stopping time of $\mathbb{F}$. We then have

\begin{align*}
|g-\tau|&=(\tau-g)^++(\tau-g)^-\\
&=(\tau-g)^++g-(\tau \wedge g) \\
&=\int_0^{\tau} \I_{\{g\leq s \}}ds+g-\int_0^{\tau} \I_{\{g>s \}}ds\\
&=\int_0^{\tau} \I_{\{g\leq s \}}ds+g-\int_0^{\tau} [1-\I_{\{g\leq s \}}]ds\\
&=g+ \int_0^{\tau}[2\I_{\{ g\leq s\}}-1]ds.
\end{align*}
From Fubini's Theorem we have

\begin{align*}
\E\left[\int_0^{\tau} \I_{\{ g\leq s\}} ds\right]&=\E\left[\int_0^{\infty} \I_{\{s<  \tau \}}\I_{\{ g\leq s\}} ds\right]\\
&=\int_0^{\infty} \E[\I_{\{s<\tau \}}\I_{\{ g\leq s\}}]ds\\
&=\int_0^{\infty}\E[ \E[\I_{\{s<\tau \}}\I_{\{ g\leq s\}}|\F_s]]ds\\
&=\int_0^{\infty}\E[\I_{\{s<\tau \}}  \E[\I_{\{ g\leq s\}}|\F_s]]ds\\
&=\E\left[ \int_0^{\infty}\I_{\{s<\tau \}}  \E[\I_{\{ g\leq s\}}|\F_s]ds \right]\\
&=\E\left[ \int_0^{\tau}  \P( g\leq s |\F_s)ds \right].
\end{align*}

Note that due to Remark \ref{rem:lastzero}, the event $\{g\leq s\}$ is equal to $\{ X_u \geq 0 \text{ for all } u\in [s,\infty)\}$ (up to a $\P$-null set). Hence, since $X_s$ is $\F_s$-measurable,
\begin{align*}
\P(g\leq s|\F_s)&=\P(X_u >0 \text{ for all } u\in [s,\infty)|\F_s )\\
&=\P\left(\inf_{u \geq s} X_u \geq 0|\F_s\right) \\
&=\P\left(\inf_{u \geq s} (X_u-X_s) \geq -X_s|\F_s\right) \\
&=\P\left(\inf_{u \geq  0} \widetilde{X}_u \geq -X_s|\F_s\right), 
\end{align*}
where $\widetilde{X}_u=X_{s+u}-X_s$ for $u\geq 0$. From the Markov property for L\'evy processes we have that $\widetilde{X}=(\widetilde{X}_u,u\geq 0)$ is a L\'evy process with the same law as $X$, independent of $\F_s$. We therefore find that
\begin{align*}
\P(g\leq s|\F_s)&=h(X_s),
\end{align*}
where $h(b)=\P(\inf_{u \geq 0} X_u \geq -b)$. 
Note that the event $\{ \inf_{u \geq 0} X_u \geq 0\}$ is equal to $\{\tau_{0}^-=\infty\}$ where $\tau_{0}^-=\inf\{s>0: X_s < 0\}$. Hence, by the spatial homogeneity of L\'evy processes
\begin{align*}
h(b)&=\P(\inf_{u \geq 0} X_u \geq -b)\\
&=\P_{b} (\inf_{u \geq 0} X_u \geq 0)\\
&=\P_b(\tau_{0}^-=\infty)\\
&=[1-\P_b(\tau_{0}^-<\infty)]\\
&=\psi'(0+) W(b),
\end{align*}
where the last equality holds by identity \eqref{eq:tau0finite} and the fact that $\psi'(0+)> 0$. Therefore,
\begin{align*}
V_*&=\inf_{\tau\in \mathcal{T}} \E(|g-\tau|)\\
&=\E(g)+\inf_{\tau \in \mathcal{T}} \left\{2 \E\left(  \int_0^{\tau} \I_{\{ g\leq s\}}ds  \right)-\E(\tau)\right\}\\
&=\E(g)+\inf_{\tau \in \mathcal{T}} \left\{2 \E\left(\int_0^{\tau} \P(g\leq s|\F_s)ds  \right)-\E(\tau)\right\}\\
&=\E(g)+\inf_{\tau \in \mathcal{T}} \left\{2 \E\left(\int_0^{\tau} h(X_s)ds  \right)-\E(\tau)\right\}\\
&=\E(g)+\inf_{\tau \in \mathcal{T}} \left\{ \E\left(\int_0^{\tau} [2h(X_s) -1]ds \right)\right\}.
\end{align*}

Hence,
\begin{align*}
V_*&=\E(g)+\inf_{\tau \in \mathcal{T}} \left\{ \E\left(\int_0^{\tau} [2\psi'(0+)W(X_s)-1]ds \right)\right\}\\
&=\E(g)+\inf_{\tau \in \mathcal{T}} \left\{ \E\left(\int_0^{\tau} G(X_s)ds \right)\right\}.
\end{align*}

\end{proof}

To find the solution of the optimal stopping problem (\ref{eq:optimalstopping}) we will expand it to an optimal stopping problem for a strong Markov process $X$ with starting value $X_0=x\in\mathbb{R}$. Specifically, we define the function $V: \R \mapsto \R$ as 
\begin{align}
\label{eq:optimalstoppingforallx}
V(x)=\inf_{\tau\in \mathcal{T} } \E_x\left( \int_0^{\tau} G(X_s)ds\right).
\end{align}
Thus, 
\begin{align*}
V_*=V(0)+\E(g).
\end{align*}

\begin{rem}
Note that the distribution function of $-\underline{X}_{\infty}$ is given by 
\begin{align*}
F(x)=\P(-\underline{X}_{\infty} \leq x)=\P_x(\tau_0^-=\infty)=\psi'(0+)W(x).
\end{align*}
Hence the function $G$ can be written in terms of $F$ as $G(x)=2F(x)-1$.

\end{rem}

Let us now give some intuition about the optimal stopping problem \eqref{eq:optimalstoppingforallx}. For this define $x_0$ as the lowest value $x$ such that $G(x)\geq 0$, i.e.

\begin{align}
\label{eq:firstzeroofG}
x_0=\inf\{x\in \R:G(x)\geq 0 \}.
\end{align}  
We know that $W$ is continuous and strictly increasing on $[0,\infty)$ and vanishes on $(-\infty,0)$. Moreover, we have that $\lim_{x\rightarrow \infty} W(x)=1/\psi'(0+)$ (since $F(x)=\psi'(0+)W(x)$ is a distribution function). As a consequence we have that $G$ is a strictly increasing and continuous function on $[0,\infty)$ such that $G(x)=-1$ for $x<0$ and $G(x)\xrightarrow{x\rightarrow \infty} 1$. In the same way as $W$, $G$ may have a discontinuity at zero depending of the path variation of $X$. From the fact that $G(x)=-1$ for $x<0$ and the definition of $x_0$ given in \eqref{eq:firstzeroofG} we have that $x_0\geq 0$.

The above observations tell us that, to solve the optimal stopping problem \eqref{eq:optimalstoppingforallx}, we are interested in a stopping time such that before stopping, the process $X$ spent most of the time in those values where $G$ is negative, taking into account that $X$ can pass some time in the set $\{ x \in \R: G(x)>0\}$ and then return back to the set $\{x \in \R: G(x)\leq 0 \}$.\\

It therefore seems reasonable to think that a stopping time which attains the infimum in \eqref{eq:optimalstoppingforallx} is of the form,

\begin{align*}
\tau_a^+=\inf\{t>0: X_t> a \}
\end{align*}
for some $a \in \R$.

 
The following theorem is the main result of this work. It confirms the intuition above and links the optimal stopping level with the median of the convolution with itself of the distribution function of $-\underline{X}_{\infty}$.

\begin{thm}
\label{thm:maintheorem}
Suppose that $X$ is a spectrally negative L\'evy process drifting to infinity with L\'evy measure $\Pi$ satisfying 
\begin{align*}
\int_{(-\infty,-1)}x^2\Pi(dx)<\infty.
\end{align*}
Then there exists some $a^*\in [x_0,\infty)$ such that an optimal stopping time in \eqref{eq:optimalstoppingforallx} is given by
\begin{align*}
\tau^*=\inf\{t\geq 0: V(X_t)=0 \}=\inf\{t\geq 0: X_t\geq a^*\}.
\end{align*}
The optimal stopping level $a^*$ is defined by 

\begin{align}
\label{eq:characterisationofa}
a^*=\inf\{x\in \R : H(x)\geq 1/2 \}
\end{align}
where $H$ is the convolution of $F$ with itself, i.e.,
\begin{align*}
H(x)=\int_{[0,x]} F(x-y)dF(y)
\end{align*}
Furthermore, $V$ is a non-decreasing, continuous function satisfying the following:

\begin{itemize}
\item[$i)$] If $X$ is of infinite variation or finite variation with 

\begin{align}
\label{eq:rhocondition}
F(0)^2 < 1/2,
\end{align}
then $a^*>0$ is the median of the distribution function $H$, i.e. is the unique value which satisfies the following equation

\begin{align}
\label{eq:aoptimal}
H(a^*)= \int_{[0,a^*]} F(a^*-y)dF(y)=\frac{1}{2}
\end{align}
The value function is given by

\begin{align}
\label{eq:valuefunctionintermsofscale}
V(x)=\frac{2}{\psi'(0+)}\int_x^{a^*} H(y)dy-\frac{a^*-x}{\psi'(0+)}\I_{\{x\leq a^* \}}
\end{align}

Moreover, there is smooth fit at $a^*$ i.e. $V'(a^*-)=0=V'(a^*+)$.

\item[$ii)$] If $X$ is of finite variation with $F(0)^2 \geq  1/2$ then $a^*=0$ and 
\begin{align*}
V(x)=\frac{x}{\psi'(0+)}\I_{\{x\leq 0 \}}.
\end{align*}

 In particular, there is continuous fit at $a^*=0$ i.e. $V(0-)=0$ and there is no smooth fit at $a^*$ i.e. $V'(a^*-)>0$.

\end{itemize}
\end{thm}

\begin{rem}
\label{rem:main:thm}
\begin{itemize}
\item[$i)$] Note that since $F$ corresponds to the distribution function of $-\underline{X}_{\infty}$, $H$ can be interpreted as the distribution function of $Z=-\underline{X}_{\infty}-\underline{Y}_{\infty}$ where $-\underline{Y}_{\infty}$ is an independent copy of $-\underline{X}_{\infty}$. Moreover, $H$ can be written in terms of scale functions as

\begin{align}
\label{eq:alternativeexpresionH}
H(x)=\psi'(0+)^2W(x)W(0)+\psi'(0+)^2\int_0^x W(y)W'(x-y)dy
\end{align}
and then equation \eqref{eq:aoptimal} reads.

\begin{align*}
\psi'(0+)^2\left[  W(a^*)W(0)+\int_0^x W(y)W'(a^*-y)dy\right]=\frac{1}{2}
\end{align*}
Using Fubini's Theorem the value function takes the form 

\begin{align}
\label{eq:alternativeexpresionV}
V(x)=\left( 2\psi'(0+) \int_0^{a^*}W(y)W(a^*-y)dy-2\psi'(0+) \int_0^{x}W(y)W(x-y)dy-\frac{a^*-x}{\psi'(0+)}\right)\I_{\{x\leq a^* \}}.
\end{align}

\item[$ii)$] 
Note that in the case that $X$ is of finite variation the condition  $F(0)^2 \geq  1/2$ is equivalent to $0 > \int_{(-\infty,0)}x\Pi(dx) /d \geq 1/\sqrt{2}-1 $ (since $F(0)=\psi'(0+)/d=1+\int_{(-\infty,0)}x\Pi(dx)/d$ and $\int_{(-\infty,0) } x\Pi(dx)\leq 0$) so the condition given in $ii)$ tells us that the drift $d$ is much larger than the average size of the jumps. This implies that the process drifts quickly to infinity and then we have to stop the first time that the process $X$ is above zero. In this case, concerning the optimal prediction problem, the stopping time which is nearest (in the $L^1$ sense) to the last time that the process is below zero is the first time that the process is above the level zero.

\item[$iii)$] If $X$ is of finite variation with $F(0)^2<1/2$ then $\int_{(-\infty,0)}x\Pi(dx) /d  < 1/\sqrt{2}-1<0$ we have that the average of size of the jumps of $X$ are sufficiently large such that when the process crosses above the level zero the process is more likely (than in $ii)$) that the process $X$ jumps again below and spend more time in the region where $G$ is negative. This condition also tells us that the process $X$ drifts a little slower to infinity that in the $ii)$. The stopping time which is nearest (in the $L^1$ sense) to the last time that the process is below zero is the first time that the process is above the level $a^*$.

\end{itemize}

\end{rem}

\section{Proof of Main Result}
\label{sec_proof}
In the next section we proof Theorem \ref{thm:maintheorem} using a direct method. Since proof is rather long, we break it into a number of lemmas.

In particular, we will use the general theory of optimal stopping (see \cite{peskir2006optimal}) to get a direct proof of Theorem \ref{thm:maintheorem}. First, using the Snell envelope we will show that an optimal stopping time for \eqref{eq:optimalstoppingforallx} is the first time that the process enters to a stopping set $D$, defined in terms of the value function $V$. Recall the set 

\begin{align*}
\mathcal{T}_t=\{\tau \geq t: \tau \text{ is a stopping time}\}.
\end{align*}

We denote $\mathcal{T}=\mathcal{T}_0$ as the set of all stopping times.

The next Lemma is standard in optimal stopping and we include the proof for completeness.

\begin{lemma}
\label{lemma:Optimaltau}
Denoting by $D=\{x\in \R:V(x)=0 \}$ the stopping set, we have that for any $x\in \R$ the stopping time 
\begin{align*}
\tau_{D}=\inf\{t\geq 0: X_t\in D \}
\end{align*}
attains the infimum in $V(x)$, i.e. $V(x)= \E_x\left( \int_0^{\tau_D} G(X_s)ds\right)$.
\end{lemma}

\begin{proof}
From the general theory of optimal stopping consider the Snell envelope defined as

\begin{align*}
S_t^x=\essinf_{\tau \in \mathcal{T}_t} \E\left(\int_0^{\tau} G(X_s+x)ds\bigg|\F_t\right)
\end{align*}
and define the stopping time 

\begin{align*}
\tau^*_x=\inf\left\{t\geq 0: S_t^x=\int_0^t G(X_s+x)ds \right\}.
\end{align*}
Then we have that the stopping time is $\tau^*_x$ is optimal for 

\begin{align}
\label{eq:axuliaryresult10}
\inf_{\tau \in \mathcal{T}} \E\left(\int_0^{\tau} G(X_s+x)ds \right).
\end{align}
On account of the Markov property we have 

\begin{align*}
S_t^x&=\essinf_{\tau \in \mathcal{T}_t} \E\left(\int_0^{\tau} G(X_s+x)ds\bigg|\F_t\right)\\
&=\int_0^t G(X_s+x)ds+\essinf_{\tau \in \mathcal{T}_t} \E\left(\int_0^{\tau} G(X_s+x)ds-\int_0^t G(X_s+x)ds\bigg|\F_t\right)\\
&=\int_0^t G(X_s+x)ds+\essinf_{\tau \in  \mathcal{T}_t} \E\left(\int_t^{\tau} G(X_s+x)ds\bigg|\F_t\right)\\
&=\int_0^t G(X_s+x)ds+\essinf_{\tau \in  \mathcal{T}_t} \E\left(\int_0^{\tau-t} G(X_{s+t}+x)ds\bigg|\F_t\right)\\
&=\int_0^t G(X_s+x)ds+\essinf_{\tau \in \mathcal{T}} \E_{X_t}\left(\int_0^{\tau} G(X_{s}+x)ds\right)\\
&=\int_0^t G(X_s+x)ds+V(X_t+x),
\end{align*}
where the last equality follows from the spatial homogeneity of L\'evy processes and from the definition of $V$. Therefore $\tau_x^*=\inf\{t\geq 0: V(X_t+x)=0 \}$. So we have
\begin{align*}
\tau_x^*=\inf\{t\geq 0: X_t+x \in D\}
\end{align*}
Thus

\begin{align*}
V(x)&=\inf_{\tau \in \mathcal{T}}\E_x\left( \int_0^{\tau} G(X_t)dt\right)\\
&=\inf_{\tau \in  \mathcal{T}}\E\left( \int_0^{\tau} G(X_t+x)dt\right)\\
&=\E\left( \int_0^{\tau_{x}^*} G(X_t+x)dt\right)\\
&=\E_x\left(\int_0^{\tau_{D}} G(X_t)dt \right),
\end{align*}
where the third equality holds since $\tau_x^*$ is optimal for \eqref{eq:axuliaryresult10} and the fourth follows from the spatial homogeneity of L\'evy processes. Therefore the stopping time $\tau_{D}$ is the optimal stopping time for $V(x)$ for all $x\in \R$.
\end{proof}

Next, we will prove that $V(x)$ is finite for all $x\in \R$ which implies that there exists a stopping time $\tau_*$ such that the infimum in \eqref{eq:optimalstoppingforallx} is attained. Recall the definition of $x_0$ in (\ref{eq:firstzeroofG}).

\begin{lemma}
The function $V$ is non-decreasing with $V(x) \in (-\infty,0]$ for all $x\in \R$. In particular, $V(x)<0$ for any $x \in (-\infty,x_0)$.
\end{lemma}

\begin{proof}
From the spatial homogeneity of L\'evy processes,
\begin{align*}
V(x)=\inf_{\tau \in \mathcal{T}}\E \left( \int_0^{\tau} G(X_s+x)ds\right).
\end{align*}
Then, if $x_1\leq x_2$ we have $G(X_s+x_1)\leq G(X_s+x_2)$ since $G$ is a non-decreasing function (see the discussion before Theorem \ref{thm:maintheorem}). This implies that $V(x_1)\leq V(x_2)$ and $V$ is non-decreasing as claimed. If we take the stopping time $\tau \equiv 0$, then for any $x\in \R$ we have $V(x)\leq 0$. Let $x<x_0$ and let $y_0 \in (x,x_0)$ then $G(x)\leq G(y_0)<0$ and from the fact that for all $s<\tau_{y_0}^+$, $X_s\leq y_0$ we have
\begin{align*}
V(x)  \leq \E_x \left( \int_0^{\tau_{y_0}^+} G(X_s)ds\right) \leq  \E_x \left( \int_0^{\tau_{y_0}^+} G(y_0)ds\right)=G(y_0)\E_x(\tau_{y_0}^+)<0,
\end{align*}
where the last inequality holds due to $\P_x(\tau_{y_0}^+>0)>0$ and then $\E_x(\tau_{y_0}^+)>0$.

Now we will see that $V(x)>-\infty$ for all $x\in \R$. Note that $G(x) \geq -\I_{\{ x\leq x_0 \}}$ holds for all $x\in \R$ and thus

\begin{align*}
V(x)&=\inf_{\tau\in \mathcal{T}} \E_x\left(\int_0^{\tau} G(X_s)ds \right)\\
& \geq \inf_{\tau\in \mathcal{T}} \E_x \left(\int_0^{\tau} -\I_{\{ X_s \leq x_0 \}} ds \right)\\
& = - \sup_{\tau\in \mathcal{T}} \E_x \left(\int_0^{\tau} \I_{\{ X_s \leq x_0 \}} ds \right)\\
& \geq -\E_x \left( \int_0^{\infty} \I_{\{ X_s \leq x_0 \}} ds \right)\\
&\geq -\E_x(g_{x_0}),
\end{align*}
%
%
where the last inequality holds since if $s>g_{x_0}$ then $\I_{\{X_s \leq x_0\}}=0$. 

%

From Lemma \ref{cor:moments} we have that $\E_x(g_{x_0})<\infty$. Hence for all $x<x_0$ we have $V(x)\geq -\E_x(g_{x_0})>-\infty$ and due to the monotonicity of $V$, $V(x)>-\infty$ for all $x\in \R$.

\end{proof}

Next, we derive some properties of $V$ which will be useful to find the form of the set $D$.

\begin{lemma}
\label{lemma:vzero}
The set $D$ is non-empty. Moreover, there exists an $\widetilde{x}$ such that
 
\begin{align*}
V(x)=0 \qquad \text{for all } x \geq \widetilde{x}.
\end{align*} 
\end{lemma}

\begin{proof}
Suppose that $D=\emptyset$. Then by Lemma \ref{lemma:Optimaltau} the optimal stopping time for \eqref{eq:optimalstoppingforallx} is $\tau_D=\infty$. This implies that 

\begin{align*}
V(x)=\E_x\left(\int_0^{\infty} G(X_t)dt\right).
\end{align*}
Let $m$ be the median of $G$, i.e.

\begin{align*}
m=\inf\{ x\in \R:G(x) \geq 1/2\}
\end{align*}
and let $g_m$ the last time that the process is below the level $m$ defined in \eqref{eq:lastzero}. Then

\begin{align}
\label{eq:Disnonemptyauxliaryresult}
\E_x\left(\int_0^{\infty} G(X_t)dt\right)&=\E_x\left(\int_0^{g_m} G(X_t)dt\right)+\E_x\left(\int_{g_m}^{\infty} G(X_t)dt\right).
\end{align}
Note that from the fact that $G$ is finite and $g_m$ has finite expectation (see Lemma \ref{cor:moments})) the first term on the right-hand side of \eqref{eq:Disnonemptyauxliaryresult} is finite. Now we analyse the second term in the right-hand side of \eqref{eq:Disnonemptyauxliaryresult}. With $n \in \N$, since $G(X_t)$ is non-negative for all $t\geq g_m$ we have  

\begin{align*}
\E_x\left(\int_{g_m}^{\infty} G(X_t)dt\right)&=\E_x\left(\I_{\{ g_m<n\}}\int_{g_m}^{\infty} G(X_t)dt\right)+\E_x\left(\I_{\{ g_m\geq n\}}\int_{g_m}^{\infty} G(X_t)dt\right)\\
&\geq \E_x\left(\I_{\{ g_m<n\}}\int_{g_m}^{n} G(X_t)dt\right)\\
& \geq \frac{1}{2} \E_x\left(\I_{\{ g_m<n\}} (n-g_m) \right).
\end{align*}

Then letting $n\rightarrow \infty$ and using the monotone convergence theorem we deduce that $V(x)=\infty$ which leads to a contradiction. From the fact that $V$ is a non-decreasing function and the set $D\neq \emptyset$ we have that there exists a $\widetilde{x}$ sufficiently large such that $V(x)=0$ for all $x\geq \widetilde{x}$.

%
%
%
%
%

\end{proof}

\begin{lemma}
\label{lemma:continuityofV}
The function $V$ is continuous. 
\end{lemma}

\begin{proof}
From the previous lemma we know that there exists an $\widetilde{x}$ such that $V(x)=0$ for all $x\geq \widetilde{x}$. As $X$ is a spectrally negative L\'evy process drifting to infinity we have that $X_{\tau_{\widetilde{x}}^+}=\widetilde{x}$ $\P$-a.s. and thus

\begin{align*}
V(x)&=\inf_{\tau \in \mathcal{T}} \E_x\left( \int_0^{\tau} G(X_t)dt \right)\\
&=\inf_{\tau \in \mathcal{T}} \E_x\left( \I_{\{ \tau<\tau_{\widetilde{x}}^+\}}\int_0^{\tau} G(X_t)dt+\I_{\{ \tau\geq \tau_{\widetilde{x}}^+\}}\int_0^{\tau_{\widetilde{x}}^+}G(X_t)dt+\I_{\{ \tau\geq \tau_{\widetilde{x}}^+\}}\int_{\tau_{\widetilde{x}}^+}^{\tau} G(X_t)dt \right)\\
&=\inf_{\tau \in \mathcal{T}} \E_x\left( \int_0^{\tau \wedge \tau_{\widetilde{x}}^+}G(X_t)dt+\I_{\{ \tau\geq \tau_{\widetilde{x}}^+\}}\int_{0}^{\tau-\tau_{\widetilde{x}}^+} G(X_{t+\tau_{\widetilde{x}}^+})dt \right)\\
&=\inf_{\tau \in \mathcal{T}} \E_x\left( \E_x\left( \int_0^{\tau \wedge \tau_{\widetilde{x}}^+}G(X_t)dt+\I_{\{ \tau\geq \tau_{\widetilde{x}}^+\}}\int_{0}^{\tau-\tau_{\widetilde{x}}^+} G(X_{t+\tau_{\widetilde{x}}^+})dt \bigg| \F_{\tau_{\widetilde{x}}^+}\right) \right)
\end{align*}
Using the strong Markov property of $X$ and the fact that $V(\widetilde{x})=0$ we have 

\begin{align*}
V(x)
&=\inf_{\tau \in \mathcal{T}} \E_x\left(  \int_0^{\tau \wedge \tau_{\widetilde{x}}^+}G(X_t)dt+\I_{\{ \tau\geq \tau_{\widetilde{x}}^+\}} \E_{X_{\tau_{\widetilde{x}}^+}}\left(\int_{0}^{\tau} G(X_{t+\tau_{\widetilde{x}}^+})dt \right) \right)\\
&=\inf_{\tau \in \mathcal{T}} \E_x\left(  \int_0^{\tau \wedge \tau_{\widetilde{x}}^+}G(X_t)dt+\I_{\{ \tau\geq \tau_{\widetilde{x}}^+\}} V(\widetilde{x}) \right)\\
&=\inf_{\tau \in \mathcal{T}} \E_x\left(  \int_0^{\tau \wedge \tau_{\widetilde{x}}^+}G(X_t)dt \right).
\end{align*} 
Note that the process $\left\{\int_0^t G(X_s)ds,t\geq 0\right\}$ is continuous. Then we have that 
\begin{align*}
\E_x\left( \sup_{t\geq 0} \left| \int_0^{t\wedge \tau_{\widetilde{x}}^+} G(X_s)ds\right| \right) &\leq \E_x\left( \sup_{t\geq 0}  \int_0^{t\wedge \tau_{\widetilde{x}}^+} |G(X_s)|ds \right)\\
&\leq \E_x\left(\sup_{t\geq 0} t\wedge \tau_{\widetilde{x}}^+\right)\\
&=\E_x(\tau_{\widetilde{x}}^+)\\
&<\infty,
\end{align*}
where the last quantity is finite since we know that for a spectrally negative L\'evy process
\begin{align*}
\E(e^{-q\tau_x^+})=e^{-\Phi(q)x}
\end{align*}
and then calculating derivatives with respect to $q$ and evaluating at zero (see \cite{feller1971an} (section XIII.2)) we obtain that 

\begin{align*}
 \E(\tau_x^+)=\frac{x}{\psi'(0+)}<\infty.
\end{align*}

Then from the general theory of optimal stopping time we have that the infimum in 
\begin{align}
\label{eq:optimawedge}
V(x)=\inf_{\tau \in \mathcal{T}} \E_x\left(  \int_0^{\tau \wedge \tau_{\widetilde{x}}^+}G(X_t)dt \right)
\end{align}
is attained, say $\tau_{x}^*=\inf\{t\geq 0:X_t+x \in D \}$. Note that from the definition of $\tau_x^*$ and $\tau_{\widetilde{x}}^+$ we have that $\tau_{x}^* \leq \tau_{\widetilde{x}}^+$. Now we check the continuity of $V$. As $W$ is continuous in $[0,\infty)$ we have that $W$ is uniformly continuous in the interval $[0,\widetilde{x}]$. Now take $\varepsilon>0$, then there exists some $\delta>0$ such that for all $x,y \in [0,\widetilde{x}]$ it holds $|W(x)-W(y)|<\varepsilon$ when $|x-y|<\delta$. Then we have

\begin{align*}
V(x+\delta)-V(x)& \leq \E\left( \int_0^{\tau^*_{x}}G(X_t+x+\delta)dt \right)-\E\left( \int_0^{\tau^*_{x}}G(X_t+x)dt \right)\\
&=2 \psi'(0) \E\left(  \int_0^{\tau^*_{x}} [W(X_t+x+\delta)- W(X_t+x)]dt\right)\\
& \leq 2 \psi'(0) \E\left(  \int_0^{\tau_{\widetilde{x}}^+ } [W(X_t+x+\delta)- W(X_t+x)]dt\right),
\end{align*}
where the first inequality holds since $\tau_{x}^*$ is not necessarily optimal for $V(x+\delta)$ and the last inequality follows since $W(X_t+x+\delta)- W(X_t+x)$ is always positive and from $\tau_x^*\leq \tau_{\widetilde{x}}^+$.

Recall that we have a possible discontinuity for $W$ in zero, and $W(x)=0$ for $x< 0$. Thus

\begin{align}
\label{eq:Vcontinuity}
V(x+\delta)-V(x) &\leq 2 \psi'(0) \E\left(  \int_0^{\tau_{\widetilde{x}}^+ } [W(X_t+x+\delta)- W(X_t+x)]dt\right)\nonumber\\
&=2 \psi'(0) \E\left(  \int_0^{\tau_{\widetilde{x}}^+ } \I_{\{ X_t+x+\delta<0 \}}[W(X_t+x+\delta)- W(X_t+x)]dt\right.\nonumber\\
& \qquad + \int_0^{\tau_{\widetilde{x}}^+ } \I_{\{ X_t+x+\delta\geq 0,X_t+x< 0 \}}[W(X_t+x+\delta)- W(X_t+x)]dt\nonumber\\
& \qquad + \left. \int_0^{\tau_{\widetilde{x}}^+ } \I_{\{ X_t+x\geq 0 \}} [W(X_t+x+\delta)- W(X_t+x)]dt\right).
\end{align}
Note that the first term in \eqref{eq:Vcontinuity} is zero. Now we analyse the second term. Using the monotonicity of $W$, Fubini's Theorem and the density of the potential measure of $X$ killed on exiting $(-\infty,\tilde{x})$ given in  \eqref{eq::qpotentialkillinglessa} we have
\begin{align*}
\E&\left( \int_0^{\tau_{\widetilde{x}}^+ } \I_{\{ X_t+x+\delta\geq 0,X_t+x< 0 \}}[W(X_t+x+\delta)- W(X_t+x)]dt\right)\\
&\leq W(\widetilde{x}+x+\delta)\E\left( \int_0^{\tau_{\widetilde{x}}^+ } \I_{\{ X_t+x+\delta\geq 0,X_t+x< 0 \}}dt\right)\\
&= W(\widetilde{x}+x+\delta) \int_0^{\infty } \P_x(-\delta\leq X_t<  0 ,\tau_{\widetilde{x}}^+>t)dt\\
&= W(\widetilde{x}+x+\delta) \int_{-\delta}^0 [W(\widetilde{x}-y) -W(x-y)]dy\\
&\leq \frac{ \delta}{\psi'(0+)^2},
\end{align*}
where the final inequality follows since $W$ is strictly increasing and $\lim_{x\rightarrow \infty}W(x)=1/\psi'(0+)$. Finally we inspect the third term in \eqref{eq:Vcontinuity}, using the finiteness of the moment of $\tau_{\widetilde{x}}^+$ we obtain

\begin{align*}
\E& \left( \int_0^{\tau_{\widetilde{x}}^+ } \I_{\{ X_t+x\geq 0 \}} [W(X_t+x+\delta)- W(X_t+x)]dt\right)<  \varepsilon \E \left( \int_0^{\tau_{\widetilde{x}}^+ } \I_{\{ X_t+x\geq 0 \}} dt\right) \leq \varepsilon \E(\tau_{\widetilde{x}}^+)<\infty.
\end{align*}

Hence

\begin{align*}
V(x+\delta)-V(x)<\frac{ \delta}{\psi'(0+)^2}+\varepsilon \E(\tau_{\widetilde{x}}^+)
\end{align*}

and the continuity holds.
\end{proof}

From Lemmas \ref{lemma:vzero} and \ref{lemma:continuityofV} we have that the set $D=\{ x:V(x)=0\}=[a,\infty)$ for some $a\in \R_+$. From Lemma \ref{lemma:Optimaltau} we know that for some $a \in \R_+$
\begin{align*}
\tau_D=\inf\{ t>0:X_t \in [a,\infty]=\{ t>0:X_t \geq a\}
\end{align*}
attains the infimum in $V(x)$. As $X$ is a spectrally negative L\'evy process we have that  $\tau_D=\tau_a^+$ $\P$-a.s. and hence $\tau_a^+$ is an optimal stopping time for \eqref{eq:optimalstoppingforallx} for some $a\in \R_+$. Then we just have to find the value of $a$ which minimises the right hand side of \eqref{eq:optimalstoppingforallx} with $\tau=\tau_a^+$. So in what follows we will analyse the function  
\begin{align}
\label{eq:Vcandidate}
V_a(x)=\E_x\left(\int_0^{\tau_a^+} G(X_t) dt\right).
\end{align}
and find the value $a^*$ which minimises the function $a \mapsto V_a(x)$ for a fixed $x\in \R$. We could then conclude that  $V_{a^*}(x)=V(x)$ and $\tau_{a^*}^+$ is an optimal stopping time.

Using the Wiener--Hopf factorisation we find an a explicit form of $V_a$ in terms of the convolution of the function $F$ with itself.

\begin{lemma}
\label{lemma:formofV}
For $x\geq  a$, $V_a(x)=0$ and for $x< a$,

\begin{align}
\label{eq:Vcandidateexpression1}
V_a(x)
&=\frac{2}{\psi'(0+)}\int_{x}^{a} \int_{[0,y]} F(y-z) F(dz)dy-\frac{a-x}{\psi'(0+)}.
\end{align}

\end{lemma}

\begin{proof}

It is clear that $V_a(x)=0$ for $x\geq a$, since if the process begins above the level $a$, then the first passage time above $a$ is zero and the integral inside of the expectation in \eqref{eq:Vcandidate} is again zero. Now suppose that $x<a$, then using Fubini's Theorem twice,

\begin{align*}
V_a(x)&=\E_x\left( \int_0^{\tau_a^+}G(X_t)dt\right)\\
&=\E_x\left( \int_0^{\infty} G(X_t)\I_{\{\tau_a^+>t\}}dt\right)\\
&= \int_0^{\infty}\E_x( G(X_t)\I_{\{\tau_a^+>t\}})dt\\
&= \int_0^{\infty}\E_x( G(X_t)\I_{\{ \overline{X}_t < a  \}})dt
\end{align*}
where $\overline{X}_t =\sup_{0\leq s \leq t} X_s$ is the running supremum. Denote by $e_q$ an exponential distribution with parameter $q>0$ independent of $X$, then using the dominated convergence theorem we obtain
\begin{align*}
V_a(x)&=\int_0^{\infty}\E_x( G(X_t)\I_{\{ \overline{X}_t < a  \}})dt\\
&=\lim_{q\downarrow 0}  \int_0^{\infty} e^{-qt} \E_x( G(X_t)\I_{\{ \overline{X}_t < a  \}})dt\\
&=\lim_{q\downarrow 0} \frac{1}{q} \E( G(X_{e_q}+x)\I_{\{ \overline{X}_{e_q} < a-x  \}})\\
&=\lim_{q\downarrow 0} \frac{1}{q} \E( G(-(\overline{X}_{e_q}-X_{e_q})+\overline{X}_{e_q}+x)\I_{\{ \overline{X}_{e_q} < a -x \}}).
\end{align*}
From the Wiener--Hopf factorisation (see for example \cite{kyprianou2013fluctuations} Theorem 6.15) we know that $\overline{X}_{e_q}$ is independent of $\overline{X}_{e_q}-X_{e_q}$ and that $\overline{X}_{e_q}$ is exponentially distributed with parameter $\Phi(q)$. Thus we have

\begin{align*}
V_a(x)&=\lim_{q\downarrow 0} \frac{1}{q} \E( G(-(\overline{X}_{e_q}-X_{e_q})+\overline{X}_{e_q}+x)\I_{\{ \overline{X}_{e_q} < a -x \}})\\
&=\lim_{q\downarrow 0} \frac{1}{q} \int_{0}^{a-x} \E( G(-(\overline{X}_{e_q}-X_{e_q})+x+y))\P(\overline{X}_{e_q} \in dy)\\
&=\lim_{q\downarrow 0} \frac{\Phi(q)}{q} \int_{0}^{a-x} \E_x( G(-(-\underline{X}_{e_q})+x+y))e^{-\Phi(q)y}dy\\
&=\lim_{q\downarrow 0} \frac{\Phi(q)}{q} \int_{0}^{a-x} \int_{[0,\infty)} G(x+y-z) \P(-\underline{X}_{e_q} \in dz)e^{-\Phi(q)y}dy,
\end{align*}
where the last equality follows since $\overline{X}_{t}-X_t\stackrel{d}{=} -\underline{X}_t$ for all $t\geq 0$. From the expression of the density of $-\underline{X}_{e_q}$ given in equation \eqref{eq:densityofrunninginfimum} we deduce that 

\begin{align*}
V_a(x)&=\lim_{q\downarrow 0} \frac{\Phi(q)}{q} \int_{0}^{a-x} \int_{[0,\infty)} G(x+y-z) \left(\frac{q}{\Phi(q)}W^{(q)}(dz)-qW^{(q)}(x)dz \right)e^{-\Phi(q)y}dy\\
&=\int_{0}^{a-x} \int_{[0,\infty)} G(x+y-z) W(dz)dy\\
\end{align*}
where the last equality follows by the dominated convergence theorem, $q\mapsto W^{(q)}$ is an analytic function and $\lim_{q\rightarrow 0}\Phi(q)=0$. Now recall that $G(x)=2\psi'(0+)W(x)-1=2F(x)-1$ and that $W(x)=0$ for all $x<0$

\begin{align*}
V_a(x)&=\int_{0}^{a-x} \int_{[0,x+y]} 2F(x+y-z) W(dz)dy-(W(\infty)-W(0-))(a-x)\\
&=\frac{2}{\psi'(0+)}\int_{0}^{a-x} \int_{[0,x+y]} F(x+y-z) F(dz)dy-\frac{a-x}{\psi'(0+)}\\
&=\frac{2}{\psi'(0+)}\int_{x}^{a} \int_{[0,y]} F(y-z) F(dz)dy-\frac{a-x}{\psi'(0+)}\\
\end{align*}
and the proof is complete.
\end{proof}

Now we characterise the value at which the function $a \mapsto V_a(x)$ achieves its minimum value. Recall that $x_0$ is smallest value in which the function $G$ is positive, i.e.

\begin{align*}
x_0=\inf\{x\in \R: G(x)\geq 0 \}
\end{align*}

\begin{lemma}
\label{lemma:aoptimal}
For all $x\in \R$ the function $a\mapsto V_a(x)$ achieves its minimum value in $a^* \geq x_0$ which does not depend on the value of $x$. The value $a^*$ is characterised as in Theorem \ref{thm:maintheorem}.

\end{lemma}

\begin{proof}
We know that the function $V_a(x)$ is given for $x\leq a$
\begin{align*}
V_a(x)=\frac{2}{\psi'(0+)} \int_x^a \int_{[0,y]} F(y-z)F(dz)dy-\frac{a-x}{\psi'(0+)}
\end{align*}
then $a\mapsto V_a(x)$ is differentiable and for $x$ sufficiently small
\begin{align*}
f(a):=\frac{\partial }{\partial a} V_a(x)=\frac{2}{\psi'(0+)}  \int_{[0,a]} F(a-z)F(dz)-\frac{1}{\psi'(0+)}.
\end{align*}
Recall that $F$ is a continuous function in $[0,\infty)$ and $F(x)=0$ for all $x<0$ then we can write 

\begin{align*}
f(a)=\frac{2}{\psi'(0+)}  \int_{[0,\infty)} F(a-z)F(dz)-\frac{1}{\psi'(0+)}
\end{align*} 
%
Due to monotone convergence we see that $f$ is a continuous function in $(0,\infty)$ and 
\begin{align*}
\lim_{a \rightarrow \infty}f(a)&=\frac{2}{\psi'(0+)} \int_{[0,\infty)} F(dz)-\frac{1}{\psi'(0+)}\\
&=\frac{1}{\psi'(0+)}\\
&>0
\end{align*}
where the last equality follows since $F$ is a distribution function. Moreover we have that $f(a)=-1/\psi'(0+)<0$ for $a<0$ and

\begin{align*}
f(0+)=\lim_{a \downarrow 0} f(a)=\frac{2}{\psi'(0+)}F(0)^2-\frac{1}{\psi'(0+)}.
\end{align*}
Hence in the case that $X$ is of infinite variation we have that $F(0)=0$ and thus $f(0+)<0$. Meanwhile in the case that $X$ is of finite variation $f(0+)<0$ if and only if $F(0)^2<1/2$. Hence, if $X$ is of infinite variation or $X$ is of finite variation with $F(0)^2<1/2$ there exists a value $a^*\geq 0$ such that $f(a^*)=0$ and this occurs if and only if

\begin{align*}
 \int_{[0,a^*]} F(a^*-z)F(dz)=\frac{1}{2}.
\end{align*}
In the case that that $X$ is of finite variation with $F(0)^2\geq 1/2$ we have that $f(0+)\geq 0$ and then we define $a^*=0$. Therefore we have the following: there exists a value $a^*\geq 0$ such that for $x<a^*$, $f(x)<0$ and for $x>a^*$ it holds that $f(x)>0$. This implies that the behaviour of $a \mapsto V_a(x)$ is as follows: for $a < a^*$, $a \mapsto V_a(x)$ is a decreasing function, and for $a>a^*$, $a\mapsto V_a(x)$ is increasing. Consequently $a \mapsto V_a(x)$ reaches its minimum value uniquely at $a=a^*$. That is, for all $a \in \R$

\begin{align*}
V_a(x)\geq V_{a^*}(x) \qquad \text{for all } x\in \R.
\end{align*}
It only remains to prove that $a^*\geq x_0$. Recall that the definition of $x_0$ is 

\begin{align*}
x_0=\inf\{x\in \R :G(x)\geq  0\}=\inf\{x\in \R: F(x) \geq 1/2 \}.
\end{align*}
We know from the definition of $a^*$ that $f(a^*)\geq 0$ which implies that 

\begin{align*}
1/2 \leq \int_{[0,a*]} F(a^*-z) F(dz) \leq \int_{[0,a*]}  F(dz)=F(a^*)
\end{align*}
where in the last inequality we use that $F(x)\leq 1$. Therefore we have that $a^*\geq x_0$. 
\end{proof}

We conclude that for all $x\in \R$,

\begin{align*}
V(x)=\E_x\left( \int_0^{\tau_{a^*}^+} G(X_t)dt\right),
\end{align*}
where $a^*$ is characterised in Theorem \ref{thm:maintheorem}. All that remains is to show now is the necessary and sufficient conditions for smooth fit to hold.
\begin{lemma}
We have the following:

\begin{itemize}
\item[$i)$] If $X$ is of infinite variation or finite variation with \eqref{eq:rhocondition} then there is smooth fit at $a^*$ i.e. $V'(a^*-)=0$.

\item[$ii)$] If $X$ is of finite variation and \eqref{eq:rhocondition} does not hold then there is continuous fit at $a^*=0$ i.e. $V(0-)=0$. There is no smooth fit at $a^*$ i.e. $V'(a^*-)>0$.
\end{itemize}
\end{lemma}

\begin{proof}
From Lemma \ref{lemma:formofV} we know that $V(x)=0$ for $x\geq a^*$ and for $x\leq a^*$ 

\begin{align*}
V(x)=\frac{2}{\psi'(0+)}\int_{x}^{a^*} \int_{[0,y]} F(y-z) F(dz)dy-\frac{a^*-x}{\psi'(0+)}.
\end{align*}
Note that when $X$ is of finite variation with 

\begin{align*}
F(0)^2 \geq 1/2
\end{align*}
 we have $a^*=0$ and hence

\begin{align*}
V(x)=\frac{x}{\psi'(0+)}\I_{\{x\leq  0 \}},
\end{align*}
so $V(0-)=0=V(0+)$. The left and right derivative of $V$ at $0$ are given by

\begin{align*}
V'(0-)=\frac{1}{\psi'(0+)} \qquad  \text{and} \qquad V'(0+)=0.
\end{align*}
Therefore in this case only the continuous fit at $0$ is satisfied. If $X$ is of infinite variation or finite variation with 
\begin{align*}
F(0)^2 < 1/2
\end{align*}
we have from Lemma \ref{lemma:formofV} that $a^* > 0$. Its derivative for $x\leq a^*$ is

\begin{align*}
V'(x)=-\frac{2}{\psi'(0+)}\int_{[0,x]} F(x-z) F(dz)+\frac{1}{\psi'(0+)}.
\end{align*}
Since $a^*$ satisfies $\int_{[0,x]} F(x-z) F(dz)=1/2$ we have that 
\begin{align*}
V'(a^*-)=0=V'(a^*+)
\end{align*}
Thus we have smooth fit at $a^*$.

\end{proof}

\begin{rem}
The main result can also be deduced using a classical verification-type argument. Indeed, it is straightforward to show 
that if $\tau^*\geq 0$ is a candidate optimal strategy for the optimal stopping problem \eqref{eq:optimalstoppingforallx} and $V^*(x)=\E_x\left(\int_0^{\tau^*} G(X_t)dt \right)$, $x\in \R$, then the pair $(V^*,\tau^*)$ is a solution if
\begin{enumerate}
\item $V^*(x) \leq 0$ for all $x\in \R$,
\item the process $\displaystyle{\left\{V^*(X_t)+\int_0^t G(X_s)ds,t\geq 0 \right\}}$ is a $\P_x$-submartingale for all $x\in \R$. 
\end{enumerate}

With $\tau^*=\tau_{a^*}^+$ it can be shown that the first condition is satisfied. The submartingale property can also be shown to hold using It\^o's formula. However, the proof of this turns out to be rather more involved than the direct approach, as it requires some technical lemmas to derive the necessary smoothness of the value function, as well as the required inequality linked to the submartingale propery.

\end{rem}

\section{Examples}\label{sec_examples}
We calculate numerically (using the statistical software \cite{softwareR}) the value function $x \mapsto V_a(x)$ for some values of $a\in \R$. The models used were Brownian motion with drift, Cram\'er--Lundberg risk process with exponential claims and a spectrally negative L\'evy process with no Gaussian component and L\'evy measure given by $\Pi(dy)=e^{\beta y}(e^y-1)^{-(\beta +1)}dy, y<0$.

\subsection{Brownian motion with drift}
Let $X=\{ X_t,t\geq 0\}$ be a Brownian motion with drift, i.e., $X_t$ is of the form

\begin{align*}
X_t=\sigma B_t+\mu t,\qquad t\geq 0.
\end{align*}
where $\mu>0$. Since there is absence of positive jumps and we have a positive drift, $X$ is indeed a spectrally negative L\'evy process drifting to infinity. In this case the expressions for the Laplace exponent and scale functions are well known (see for example \cite{kyprianou2011theory}, Example 1.3). The Laplace exponent is given by

\begin{align*}
\psi(\beta)=\frac{\sigma^2}{2}\beta^2+ \mu \beta , \qquad \beta\geq 0.
\end{align*}
For $q\geq 0$ the scale function $W^{(q)}$ is 

\begin{align*}
W^{(q)}(x)=\frac{1}{\sqrt{\mu^2+2q\sigma^2}}\left(\exp\left((\sqrt{\mu^2+2q\sigma^2}-\mu)\frac{x}{\sigma^2}\right) -\exp\left(-(\sqrt{\mu^2+2q\sigma^2}+\mu)\frac{x}{\sigma^2}-\right) \right), \qquad x\geq 0.
\end{align*}
Letting $q=0$ and using that $F(x)=\psi'(0+)W(x)=\mu W(x)$ we get that
\begin{align*}
F(x)=1-\exp(-2\mu/\sigma^2 x), \qquad x\geq 0.
\end{align*}
That is, $-\underline{X}_{\infty} \sim \text{Exp}(2\mu /\sigma^2)$, which implies that $H$ corresponds to the distribution function of a $\text{Gamma}(2,2 \mu/\sigma^2)$ (see Remark \ref{rem:main:thm} $i)$). Therefore $a^*$ corresponds to the median of the aforementioned Gamma distribution. In other words, $H$ is given by
\begin{align*}
H(x)=1-\frac{2 \mu }{\sigma^2}x \exp\left(-\frac{2 \mu }{\sigma^2} x\right)- \exp\left(-\frac{2 \mu }{\sigma^2} x\right), \qquad x\geq 0
\end{align*} 
and then $a^*$ is the solution to 

\begin{align*}
1-\frac{2 \mu }{\sigma^2}x \exp\left(-\frac{2 \mu }{\sigma^2} x\right)- \exp\left(-\frac{2 \mu }{\sigma^2} x\right)=\frac{1}{2}.
\end{align*}
Moreover, $V$ is given by 

\begin{align*}
V(x)=\left\{
\begin{array}{ll}
0 & x\geq a^*\\
\frac{2}{\mu}( a^*e^{-2\mu/\sigma^2 a^*} -xe^{-2\mu/\sigma^2 x})+\frac{2\sigma^2}{\mu^2}(e^{-2\mu/\sigma^2 a^*}-e^{-2\mu/\sigma^2 x} ) +\frac{a^*-x}{\mu} & 0< x< a^*\\
\frac{2}{\mu} a^*e^{-2\mu/\sigma^2 a^*} -\frac{2\sigma^2}{\mu^2}(1-e^{-2\mu/\sigma^2 a^*})+\frac{a^*+x}{\mu} & x\leq 0
\end{array}
\right.
\end{align*}
In Figure \ref{fig:valuefunctionBMwithdrift} we sketch a picture of $V_a(x)$ defined in \eqref{eq:Vcandidateexpression1} for different values of $a$. The parameters chosen for the model are $\mu=1$ and $\sigma=1$.

\begin{figure}[hbtp]
\centering
\includegraphics[scale=.4]{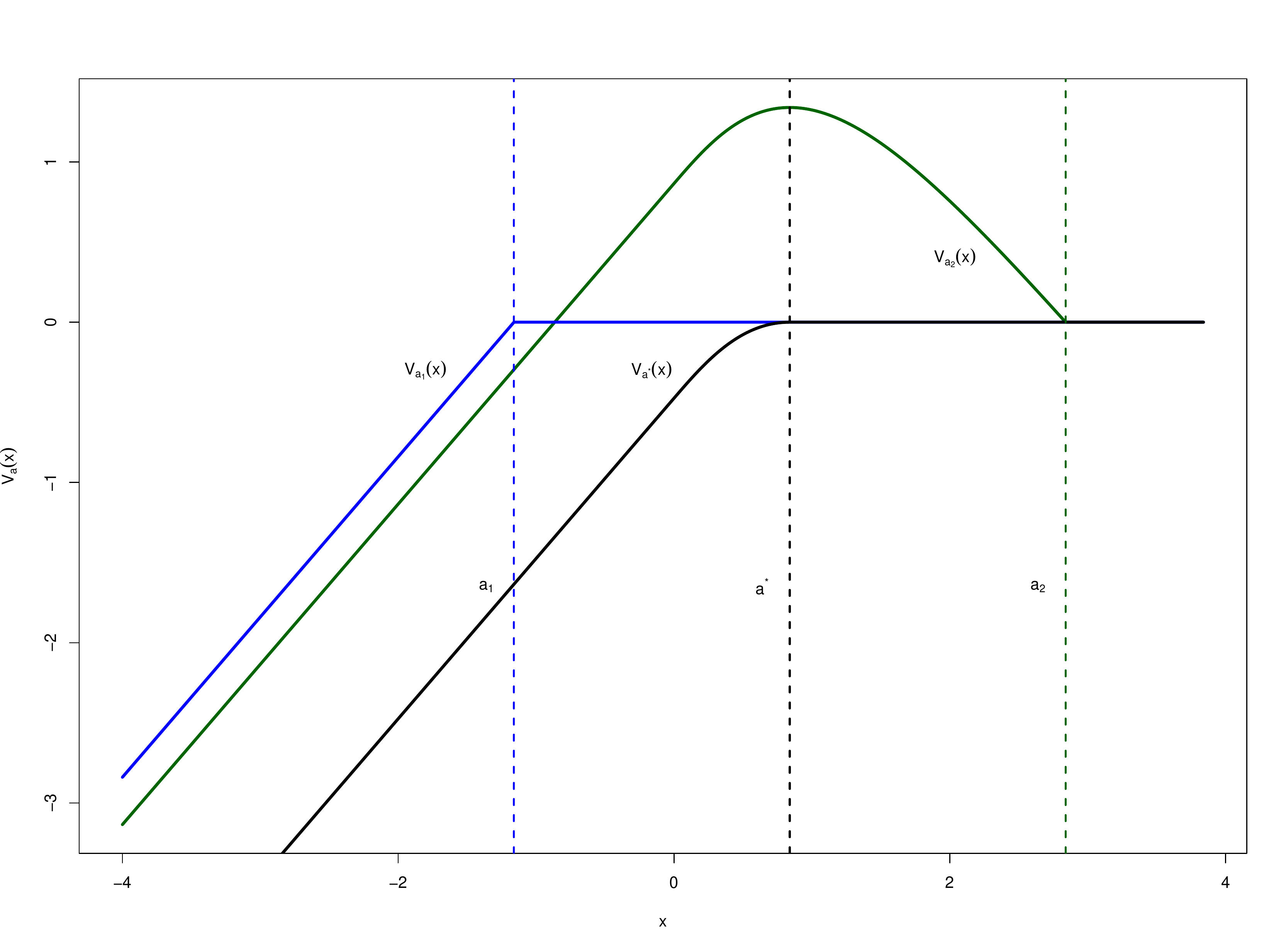}
\caption{Brownian motion with drift. Function $x\mapsto V_a$ for different values of $a$. Blue: $a<a^*$; green: $a>a^*$; black: $a=a^*$.}
\label{fig:valuefunctionBMwithdrift}
\end{figure}

\subsection{Cram\'er--Lundberg risk process}

We consider $X=\{X_t,t\geq 0\}$ as the Cram\'er--Lundberg risk process with exponential claims. That is $X_t$ is given by 

\begin{align*}
X_t=\mu t-\sum_{i=1}^{N_t}\xi_i,\qquad t\geq 0
\end{align*}
where $\mu \geq 0$, $N=\{N_t,t \geq 0 \}$ is a Poisson process with rate $\lambda \geq 0$ and $\xi_i$ is a sequence of independent and identically exponentially distributed random variables with parameter $\rho>0$. Due to the presence of only negative jumps we have that $X$ is a spectrally negative L\'evy process. It can be easily shown that $X$ is a finite variation process. Moreover, since we need the process to drift to infinity we assume that 

\begin{align*}
\frac{\lambda}{\rho \mu}<1.
\end{align*}
The Laplace exponent is given by

\begin{align*}
\psi(\beta)= \mu \beta -\frac{\lambda \beta}{\rho +\beta}, \qquad \beta\geq 0.
\end{align*}

It is well known (see \cite{hubalek2011old}, Example 1.3 or \cite{kyprianou2013fluctuations}, Exercise 8.3 iii)) that the scale function for this process is given by

\begin{align*}
W(x)=\frac{1}{\mu-\lambda/\rho}\left( 1-\frac{\lambda}{\mu \rho}\exp(-(\rho-\lambda/\mu)x)\right), \qquad x\geq 0.
\end{align*}
This directly implies that
\begin{align*}
F(x)=1-\frac{\lambda}{\mu \rho}\exp(-(\rho-\lambda/\mu)x), \qquad x\geq 0.
\end{align*}
and then $H$ is given by

\begin{align*}
H(x)&=\left(1-\frac{\lambda}{\mu \rho} \right)^2+2\frac{\lambda}{\mu \rho}\left(1-\frac{\lambda}{\mu \rho} \right)(1-\exp(-(\rho-\lambda/\mu)x))\\
&\qquad +\left(\frac{\lambda}{\mu \rho}\right)^2\left( 1-(\rho-\lambda/\mu)x\exp(-(\rho-\lambda/\mu)x)-\exp(-(\rho-\lambda/\mu)x))\right), \qquad x\geq 0.
\end{align*}
Hence, when $(1-\lambda/(\mu \rho))^2 \geq 1/2$ we have that $a^*=0$ and 

\begin{align*}
V(x)=\frac{x}{\mu-\lambda/\rho}\I_{\{x\leq 0 \}}.
\end{align*}
For the case $(1-\lambda/(\mu \rho))^2 \leq 1/2$, $a^*$ is the solution to the equation $H(x)=1/2$ and the value function is given by

\begin{align*}
V(x)=\left(\frac{2}{\mu-\lambda/\rho} \int_x^{a^*} H(y)dy-\frac{a^*-x}{\mu-\lambda/\rho}\right) \I_{\{x\leq a^* \}}.
\end{align*}
In Figure \ref{fig:valuefunctionCLriskprocess} we calculate numerically the value of $x\mapsto V_a(x)$ for the parameters $\mu=2$, $\lambda=1$ and $\rho=1$ and different values of $a$. In particular, the numerical value for the latter integral corresponds to $a=a^*$.

\begin{figure}[hbtp]
\centering
\includegraphics[scale=.4]{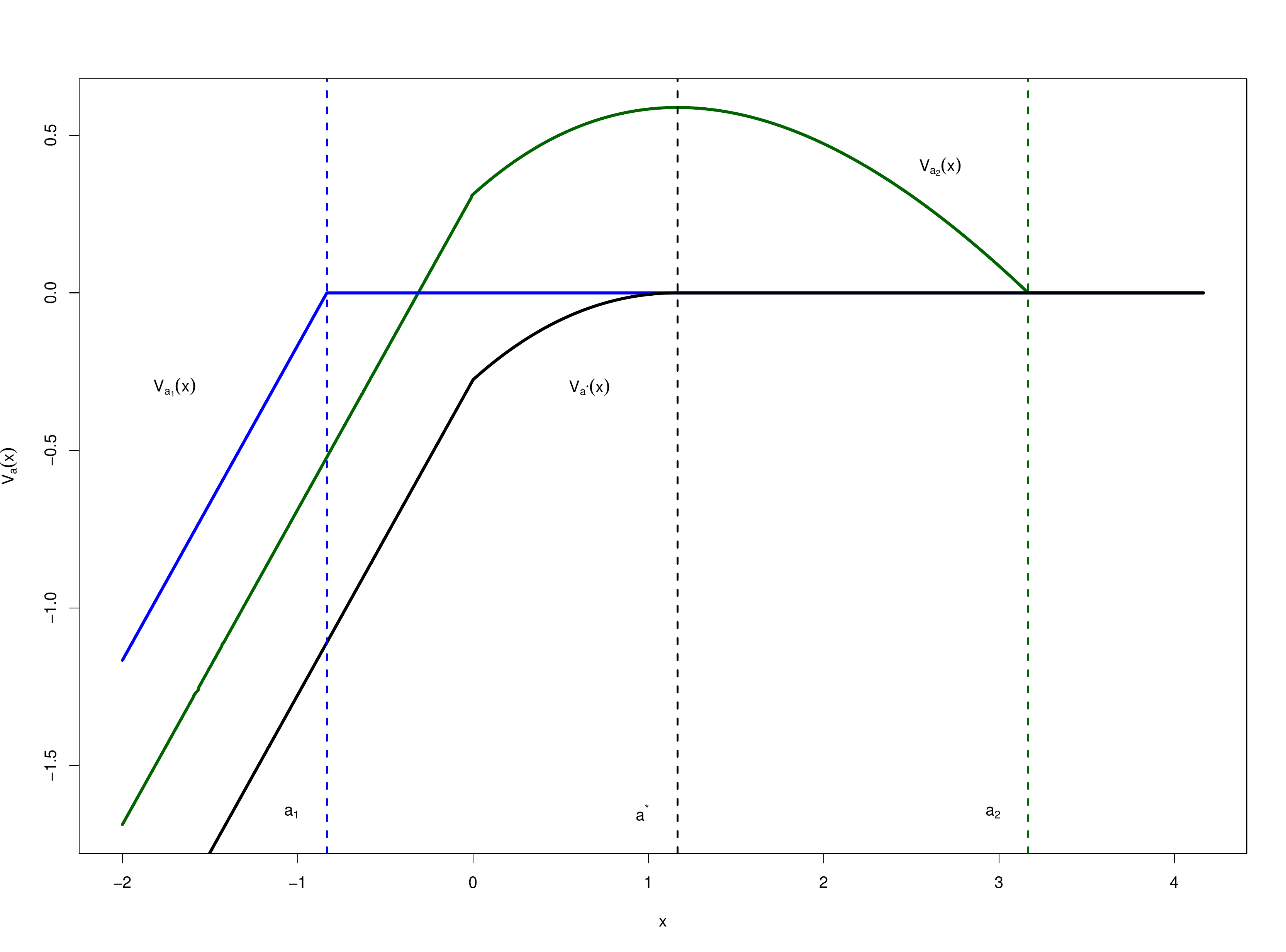}
\caption{Cram\'er--Lundberg risk process. Function $x\mapsto V_a$ for different values of $a$. Blue: $a<a^*$; green: $a>a^*$; black: $a=a^*$.}
\label{fig:valuefunctionCLriskprocess}
\end{figure}

%
%

\subsection{An infinite variation process with no Gaussian component}
In this subsection we consider a spectrally negative L\'evy process related to the theory of self-similar Markov processes and conditioned stable L\'evy processes (see \cite{chaumont2009some}). This process turns out to be the underlying L\'evy process in the Lamperti represetantion of a stable L\'evy process conditioned to stay positive. Now we describe the characteristics of this process (which can also be found in \cite{hubalek2011old}). \\

We consider $X$ as a spectrally negative L\'evy process with Laplace exponent

\begin{align*}
\psi(\theta)=\frac{\Gamma(\theta+\beta)}{\Gamma(\theta)\Gamma(\beta)} , \qquad \theta \geq 0
\end{align*} 
where $\beta \in (1,2)$. This process has no Gaussian component and the corresponding L\'evy measure is given by 

\begin{align*}
\Pi(dy)=\frac{e^{\beta y}}{(e^y-1)^{\beta+1}}dy, \qquad y<0.
\end{align*}
This process is of infinite variation and drifts to infinity. Its scale function is given by 

\begin{align*}
W(x)=(1-e^{-x})^{\beta -1}, \qquad x \geq 0.
\end{align*}
Taking the limit when $x$ goes to infinity on $W$ we easily deduce that $\psi'(0+)=1$ and thus $F(x)=W(x)$. For the case $\beta=2$ we calculate numerically the function $H$ and then the value of the number $a^*$ as well as the values of the function $V$ see Figure \ref{fig:Valuefunctionstableconditionedpositive}. We also included the values of the function $V_a(x)$ (defined in \eqref{eq:Vcandidateexpression1}) for different values of $a$ (including $a=a^*$).\\

  We close this section with a final remark. Using the empirical evidence from the previous examples we make some observations about whether the smooth fit conditions holds.
\begin{rem}
Note in Figures \ref{fig:valuefunctionBMwithdrift}, \ref{fig:valuefunctionCLriskprocess} and \ref{fig:Valuefunctionstableconditionedpositive} that the value $a^*$ is the unique value for which the function $x\mapsto V_a(x)$ exhibits smooth fit (or continuous fit) at $a^*$. When we choose $a_2>a^*$, the function $x\mapsto V_{a_2}(x)$ is not differentiable at $a_2$. Moreover, there exists some $x$ such that $V_{a_2}(x)>0$. Similarly, If $a_1<a^*$ the function $x\mapsto V_{a_1}(x)$ is also not differentiable at $a_1$. 
\end{rem}

\begin{figure}[hbtp]
\centering
\includegraphics[scale=.4]{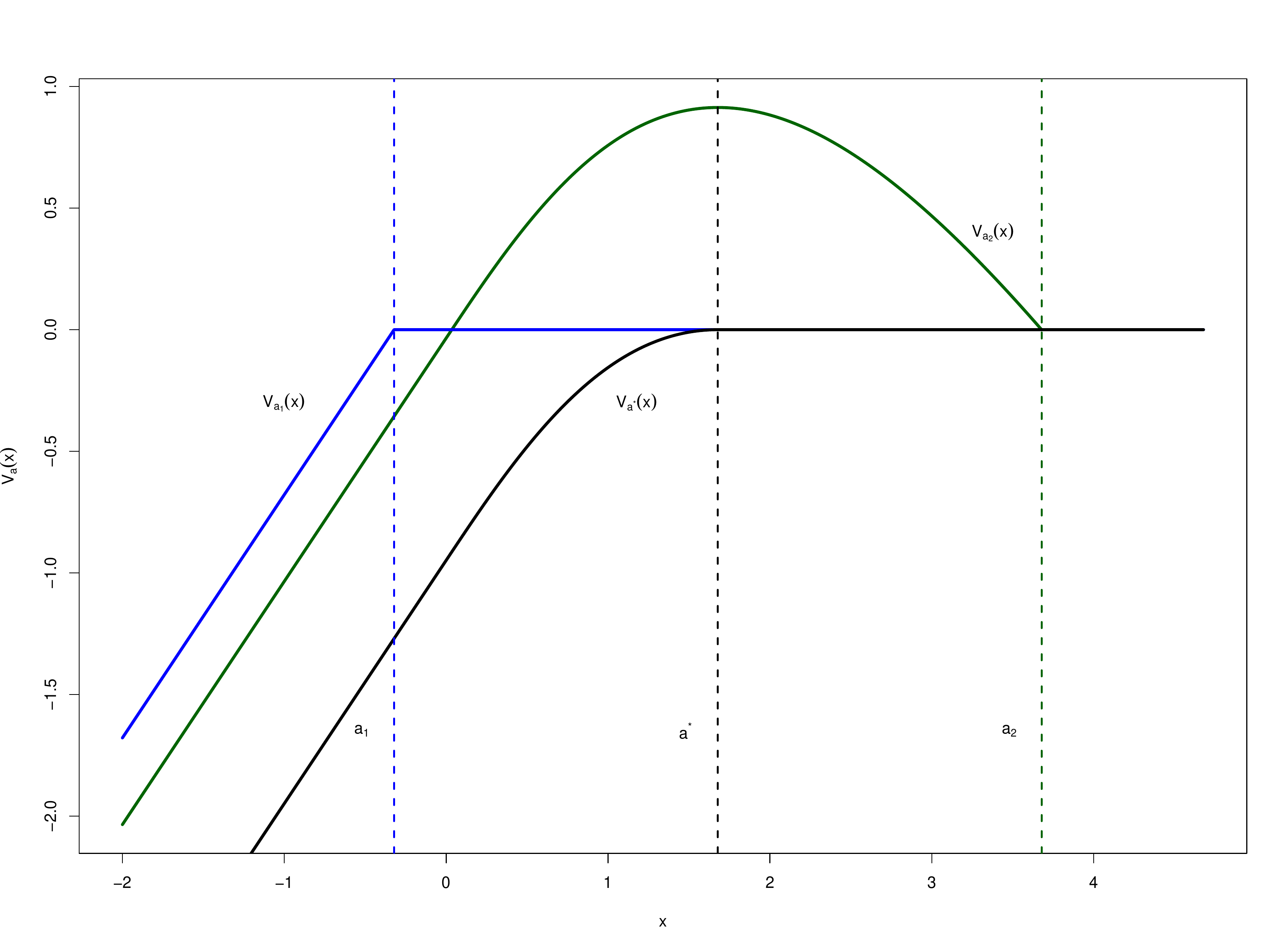}
\caption{Spectrally negative L\'evy process with L\'evy measure $\Pi(dy)=e^{2 y}(e^y-1)^{-3}dy$. Function $x\mapsto V_a$ for different values of $a$. Blue: $a<a^*$; green: $a>a^*$; black: $a=a^*$.}
\label{fig:Valuefunctionstableconditionedpositive}
\end{figure}

\bibliographystyle{apalike}

\bibliography{References}


\end{document}